\documentclass[12pt,twoside]{amsart}
\usepackage[all]{xy}
        \usepackage {amssymb,latexsym,amsthm,amsmath,mathtools, multirow, amsfonts, longtable, eqnarray, hyperref}
        \usepackage{enumitem,color}
        
\hypersetup{
	colorlinks=true,
	linkcolor=blue,
	filecolor=magenta,      
	urlcolor=cyan,
	citecolor=red,
}

\long\def\symbolfootnote[#1]#2{\begingroup%
\def\thefootnote{\fnsymbol{footnote}}\footnote[#1]{#2}\endgroup}
\newcommand{\C}{\ensuremath{\mathfrak{C}}}

\makeatletter
\def\imod#1{\allowbreak\mkern10mu({\operator@font mod}\,\,#1)}
\makeatother

\newtheorem{theorem}{Theorem}[section]
\newtheorem{lemma}[theorem]{Lemma}
\newtheorem{corollary}[theorem]{Corollary}
\newtheorem{proposition}[theorem]{Proposition}
\newtheorem*{theorem*}{Theorem}
\theoremstyle{definition}

\newtheorem{remark}[theorem]{Remark}
\newtheorem{question}[theorem]{Question}
\newtheorem{example}[theorem]{Example}

\numberwithin{equation}{section}

\newcommand{\ignore}[1]{}

\newcommand{\mynote}[1]{}
\begin{document}
	
	\title{Powers in the wreath product of $G$ with $S_n$}
	\author{Rijubrata Kundu}
	\address{IISER Pune, Dr. Homi Bhabha Road, Pashan, Pune 411 008 India}
	\email{rijubrata8@gmail.com}
	\author{Sudipa Mondal}
	\email{sudipa.mondal123@gmail.com}
	\thanks{The first named author would like to acknowledge the support of NBHM PhD fellowship during this work. The second named author would like to acknowledge the support of CSIR PhD fellowship during this work.}
	%\date{}
	\today 
	\begin{abstract}
		In this paper we compute powers in the wreath product $G\wr S_n$, for any finite group $G$. For $r\geq 2$, a prime, consider $\omega_r: G\wr S_n\to G\wr S_n$ defined by $g \mapsto g^r$. Let $P_{r}(G\wr S_n)=\frac{|\omega_r(G\wr S_n)|}{|G|^n n!}$, be the probability that a randomly chosen element in $G\wr S_n$ is a $r^{th}$ power. We prove, $P_r(G\wr S_{n+1})=P_r(G\wr S_n)$ for all $n\not \equiv -1(\text{mod } r)$ if, order of $G$ is coprime to $r$. We also give a formula for the number of conjugacy classes that are $r^{th}$ powers in $G\wr S_n$.
	\end{abstract}
	
	\maketitle
	\section{Introduction}
	
	This article deals with computing $r^{th}$ powers in the wreath product $G\wr S_n$, where $G$ is a finite group and $r\geq 2$ is an integer. These groups, in some form or the other, occur as natural subgroups of the symmetric group such as centralizers of elements, normalizers of certain subgroups and, Sylow subgroups, and as such play a important role in the theory of permutation groups. The group $C_m\wr S_n$ is called a generalised symmetric group, where $C_m$ is the cyclic group of order $m$. In particular the groups $C_2\wr S_n$ are called hyper-octahedral groups. These groups are also examples of Weyl groups. The Weyl groups of type $\mathcal{A}_n, \mathcal{B}_n,\mathcal{C}_n,\mathcal{D}_n$ are all wreath products namely, $S_{n+1}$, $C_2\wr S_n$, $C_2\wr S_n$\ and, $C_2^{n-1}\wr S_n$ respectively. This investigation of powers in the group $G\wr S_n$ is motivated from the study of powers in $S_n$, in which substantial amount of information is available. When $G$ is trivial, $G\wr S_n$ is $S_n$ itself, and therefore our results in this article generalizes some of the known results obtained for the powers in $S_n$, for example in, \cite{bl}, \cite{bmw} where the authors have studied the proportion of $r^{th}$ powers in $S_n$.
	
	Let $r\geq 2$ be any integer. The power map, $\omega_r:G\wr S_n\to G\wr S_n$ is defined by $\omega_r(g)=g^r$, for every $g\in G \wr S_n$. The elements in the image $\omega_r(G\wr S_n)=\{g^r \mid g \in G\wr S_n\}$ are called $r^{th}$ powers in $G\wr S_n$. Let $P_r(G\wr S_n):=\frac{|\omega_r(G\wr S_n)|}{|G|^nn!}$ denote the probability that a randomly chosen element of $G\wr S_n$ is a $r^{th}$ power. In this article we give characterization of $r^{th}$ powers in $G\wr S_n$, assuming $r$ is a prime, and thus obtain generating function for $P_r(G\wr S_n)$. This allows us to prove a relation that is satisfied by $P_r(G\wr S_n)$, which is one of our main results. The set of $r^{th}$ powers $\omega_r(G\wr S_n)$ is closed under conjugation, and as such is a union of conjugacy classes of $G\wr S_n$. A conjugacy class of $G\wr S_n$ contained in $\omega_r(G\wr S_n)$ is called a $r^{th}$ power conjugacy class. We count the number of conjugacy classes in $G\wr S_n$ which are $r^{th}$ powers. As stated earlier, the results proved here, brings the results obtained for the symmetric group to a more general context. It is worth mentioning that the power map along these lines has been studied for the group $\text{GL}(n,q)$, which is the group of all invertible matrices over the finite field with $q$ elements, by the first author and A. Singh in \cite{ka}.
	
	The paper is organized as follows: Section 2 collects some of the important results on powers in the symmetric group. This section not only serves as a short survey on the topic, but also allows us to point out the specific generalizations that we obtain for the powers in $G\wr S_n$. In section 3, we discuss the notion of cycle index in $G\wr S_n$, which once again is a generalization of P\'olya's cycle index in $S_n$. The notion of cycle index will play a major role in this paper. Section 4 is the main technical section where we compute the powers in $G\wr S_n$ and characterize them, which is Proposition~\ref{power-in-wreath}. In Section 5, we write generating function for the proportion of powers in $G\wr S_n$ (see Proposition~\ref{mainprobgen}), and prove one of the main results which is Theorem~\ref{mainth}. A formula for the number of $r^{th}$ power conjugacy classes is given in Theorem~\ref{powerconjformula}. We end this paper with a short discussion on the asymptotics  of the powers in $G\wr S_n$.
	
	\subsection*{Acknowledgment}
	We express our gratitude to Prof. Anupam Singh for carefully reading the paper and suggesting various improvements in the exposition of this article.
	
	\section{Powers in Symmetric Group}
	
	In this section we collect some well known results on powers in symmetric groups. Let $r\geq 2$ be an integer. Consider the power map $\omega_r:S_n \to S_n$ defined by $\omega_r(\pi)=\pi^r$, for all $\pi \in S_n$. The image of this map, $\omega_r(S_n)=\{\pi^r\mid \pi \in S_n\}$, is the set of $r^{th}$ powers in $S_n$. Let $P_r(S_n):=\frac{|\omega_r(S_n)|}{n!}$, which is the probability that a randomly chosen element in $S_n$ is a $r^{th}$ power.
		
	For a natural number $n$, $\lambda=(\lambda_1,\lambda_2, \ldots)$ where $\lambda_i\geq 0$ for all $i$ and $\lambda_1\geq \lambda_2\geq \ldots$, is called a partition of $n$, if $\sum\limits_{i}\lambda_i=n$. We denote a partition $\lambda$ of $n$ by, $\lambda \vdash n$. The positive integers $\lambda_i$ are called the parts of $\lambda$. In power notation, we write $\lambda=1^{m_1}2^{m_2}\ldots i^{m_i}\ldots$, with   $\sum\limits_{i}im_i=n$, where $m_i$ denotes the number of times $i$ occurs as a part in $\lambda$.
	The conjugacy classes in the group $S_n$ are given in terms of the cycle structure of permutations. Any permutation $\pi \in S_n$ can be written as a product of disjoint cycles, thus we define $\text{type}(\pi)=(c_1,c_2,\ldots)$, where $c_i$ denotes the number of $i$-cycles in $\pi$. It satisfies the relation  $\sum\limits_{i}ic_i=n$, and thus actually defines a partition of $n$, which in power notation is $1^{c_1}2^{c_2}\ldots \vdash n$. It is well known that two permutations $\pi, \tau \in S_n$ are conjugate if and only if $\text{type}(\pi)=\text{type}(\tau)$.  Thus, conjugacy classes in $S_n$ are in one-one correspondence with partitions of $n$.\\
	
	The $r^{th}$ powers in $S_n$ can be computed easily, if we assume that $r$ is a prime. The following lemma which can be found in \cite{bmw}, will also be useful for our computation in the wreath product. We include a proof since it will be required later.
	
	\begin{lemma}[Proposition 3.1, \cite{bmw}]\label{power-in-Sn}
	Let $r\geq 2$ be a prime. Let $\pi \in S_n$ be of $\mathrm{type}$ $(c_1,c_2,\ldots)$. Then, $\mathrm{type}(\pi^r)=(d_1,d_2,\ldots)$, where 
		\[d_i=
		\begin{cases}
		c_i+rc_{ri} & \text{ when } r\nmid i \\
		rc_{ri} & \text{ otherwise}.
		\end{cases}
		\]	
	\end{lemma}
	
	\begin{proof}
		Suppose $\alpha$ is a $k$-cycle, and $r \nmid k$. Then, it is easy to observe that $\alpha^r$ is a $k$-cycle. If $r\mid k$, say $k=ar$, then once again it is easy to see that $\alpha^r$ decomposes as a product of $r$ number of $a$-cycles. Assume now that $\pi\in S_n$ has type $(c_1,c_2,\ldots)$ and, $\pi^r$ has type $(d_1,d_2,\ldots)$. Let $1\leq i\leq n$ be such that $r\mid i$. Since, $d_i$ denotes the number of $i$-cycles in $\pi^r$, $c_{ri}$ denotes the number of $ri$-cycles in $\pi$, and the fact that each $ri$-cycle when raised to the power of $r$ gives $r$ number of $i$-cycles, it is clear that $d_i=rc_{ri}$. Observe that in the above case, the $ri$-cycles are the only ones, which contribute to the number of $i$-cycles in $\pi^r$. Now, assume that $r\nmid i$. Now, in this case, apart from the $ri$-cycles, which once again contributes $r$ number of $i$-cycles each when raised to the power $r$, the $i$-cycles itself in $\pi$, remain $i$-cycles in $\pi^r$. There are $c_i$ number of $i$-cycles in $\pi$, thus, if $r\nmid i$, $d_i=c_i+rc_{ri}$. This completes the proof.
	\end{proof}

	\noindent Using the above lemma one can easily characterize the permutations which are $r^{th}$ powers. The following proposition can be found once again in \cite{bmw}. We state it without proof.
	
	\begin{proposition}\label{power-Sn}
		Let $r\geq 2$ be a prime. Let $\pi \in S_n$ and $\mathrm{type}(\pi)$ is given by a partition $\lambda=1^{m_1}2^{m_2}\ldots i^{m_i}\ldots \vdash n$. Then $\pi \in \omega_r(S_n)$ if and only if $r \mid m_i$ whenever $r|i$.
	\end{proposition}
	
	\noindent In \cite{bmw}, the authors showed that the probability of $r^{th}$ powers  in $S_n$, that is, $P_r(S_n)$ satisfies the following,
	\begin{theorem}[\cite{bmw}, Theorem 3.4 ]\label{1}
		Let $r\geq 2$ be a prime. Then, $P_r(S_n)=P_r(S_{n+1})$, for all $n\geq 0$ and $n+1\not\equiv 0 (\text{mod }r)$. Thus, $|\omega_r(S_{n+1})|=(n+1)|\omega_r(S_n)|$, when $n\not\equiv -1(\text{mod }r)$.
	\end{theorem}
	
	\begin{example}
		Observe that Theorem~\ref{1}, doesn't hold if we take $r$ to be composite. Let $r=6$. We have $P_6(S_4)=\frac{1}{6}\neq \frac{1}{3}=P_6(S_5)$.
	\end{example}
	
	\noindent The authors in \cite{bmw}  proved this result using bijective methods. They further investigated $P_r(S_n)$ as a sequence in $n$, and proved that it is monotonically decreasing in $n$ and $\lim\limits_{n\to \infty}P_r(S_n)=0$. In other words, powers in the symmetric groups are really scarce, for large $n$.\\
	
	The special case of $r=2$ was investigated by J. Blum in \cite{bl}, using generating functions. The generating function for the powers in $S_n$ can be derived by using Polya's cycle index of $S_n$. We briefly describe the notion of cycle index of $S_n$ here, for we will see in the next section that this idea can also be generalized for $G\wr S_n$.
	
	\noindent The cycle index of $S_n$ is defined as
	\begin{equation}\label{cycleindex-S_n}
	Z_n=Z_n(t_1,t_2,\ldots t_n;S_n)=\frac{1}{n!}\sum\limits_{\substack{\pi \in S_n \\ \mathrm{type}(\pi)\\=(c_1,c_2,\ldots )}}t_1^{c_1}t_2^{c_2}\ldots t_i^{c_i}\ldots
	\end{equation}
	
	\noindent Observe that the product inside is a finite product, since there exists $m\in \mathbb{N}$, such that $c_i=0$ for all $i\geq m$. The coefficient of the monomial $t_1^{a_1}t_2^{a_2}\ldots $, where $\sum\limits_{i}ia_i=n$, is equal to $\frac{|Cl(\sigma)|}{n!}$, where $\sigma \in S_n$ is such that $\text{type}(\sigma)=(a_1,a_2,\ldots)$ and $Cl(\sigma)$ denotes the conjugacy class of $\sigma$.
	\begin{example}
		The cycle index $Z_3=Z_3(t_1,t_2,t_3;S_3)$ of $S_3$ is given by $\frac{1}{6}(t_1^3+3t_1t_2+2t_3)$.
	\end{example}

	The cycle index generating function is given by
	\begin{equation}\label{cycleindex-genfunction-S_n}
	 \displaystyle 1+\sum_{n=1}^{\infty}Z_nu^n=\prod_{i\geq 1}\mathrm{exp}\left(\frac{u^i}{i}\right).
	\end{equation}
	
	\noindent Using this, one can write the generating function for the proportion of square permutations (see Equation 5, \cite{bl}) which is,
	\begin{equation}\label{squares-S_n}
	1+\sum_{n=1}^{\infty}P_2(S_n)u^n=\left(\frac{1+u}{1-u}\right)^{1/2}\prod_{k=1}^{\infty}\cosh{\left(\frac{u^{2k}}{2k}\right)}.
	\end{equation}
	
	\noindent Blum used analytic properties of this generating function, to give asymptotic estimate of $P_2(S_n)$. He proved, $P_2(S_n)\sim K\sqrt{\frac{2}{\pi}}n^{-\frac{1}{2}}$, where $K=\prod\limits_{k=1}^{\infty}\cosh(\frac{1}{2k})$.
	
	In \cite{po}, the author investigated the powers in $S_n$ for any $r\geq 2$. He generalized Blum's estimate for $P_2(S_n)$ as,
	$$P_r(S_n)\sim_{n \to \infty} \frac{\pi_r}{n^{1-\varphi(r)/r}} $$
	where $\varphi$ denotes the Euler's phi function and $\pi_r$, an explicit constant. In  \cite{bghp}, the authors have determined the constant $\pi_r$ more explicitly. For more results and estimates on powers in $S_n$ and related ideas, we urge the reader to look at \cite{be}, \cite{ba}, \cite{bb}, \cite{mp}, \cite{pa} and, \cite{tu}. The squares in the alternating group $A_n$ is determined in \cite{pou}.

	\section{Cycle index of the wreath product $G\wr S_n$}
	In this section, we will briefly discuss the notion of cycle index for the wreath product $G\wr S_n$, for a finite group $G$, as indicated in \cite{pr}. As mentioned before this is a generalization of the cycle index for $S_n$ (see Section 2, Equation~\ref{cycleindex-S_n}). Using the cycle index, we will obtain generating functions for powers in $G\wr S_n$ in Section 5.\\

	We start by briefly describing the group $G\wr S_n$, in the form we will use. We follow the exposition in \cite{ke}. Let $G$ be a finite group and $H$ be a permutation group on $\Omega=\{1,2,\ldots, n\}$. Define,
	\begin{center}
		$G\wr H=\{(f,\pi) \mid f: \Omega\to G, \pi \in H \}$
	\end{center}
	 where $f:\Omega\to G$ is any function. Define multiplication on  $G\wr H$ as,
	\begin{center}
		$(f,\pi).(f^{'}, \pi^{'})=(ff^{'}_{\pi}, \pi\pi^{'})$
	\end{center}
	 where $f_\pi:\Omega\to G$ is defined by $f_{\pi}(i)=f(\pi^{-1}(i))$ for every $i\in \Omega$, and for two functions $f,g:\Omega\to G$, $fg:\Omega\to G$ is defined by $(fg)(i)=f(i)g(i)$. Therefore, $ff^{'}_{\pi}:\Omega\to G$ is defined by $ff^{'}_{\pi}(i)=f(i)f^{'}(\pi^{-1}i)$ for every $i\in \Omega$. The set $G\wr H$ with the above binary operation forms a group, with $(id,1_H)$ as the identity, where $id:\Omega\to G$ is defined by $id(i)=1_G$ for every $i\in \Omega$. For $(f,\pi) \in G\wr H$,  $(f^{-1}_{\pi^{-1}}, \pi^{-1})$ is its inverse, where $f^{-1}:\Omega \to G$ is defined by $f^{-1}(i)=f(i)^{-1}$ for every $i\in \Omega$. We have, $f^{-1}_{\pi^{-1}}=(f_{\pi^{-1}})^{-1}$. This completes the description of the group $G\wr H$, which is called $G$ wreath $H$. We have, $|G\wr H|=|G|^n|H|$. The above description also shows that $G\wr H$ is the semidirect product of $G^n=\underbrace{G\times G\times \ldots \times G}_{n \text{ times}}$ with $H$, under the automorphism $\phi:H\to \text{Aut}(G^n)$, where $\phi_h$ acts on $G^n$ by changing indices according to $h$. \\
	 
	 For the rest of the paper, we take $H=S_n$. To define the cycle index of $G\wr S_n$, we need the description of conjugacy classes in $G\wr S_n$. We have already seen in section 2, that conjugacy classes in $S_n$ are characterized by the type of a permutation, which is just a partition of $n$. One can generalize this notion of the type of a permutation to understand the conjugacy classes in $G\wr S_n$ (see \cite{ke}). We describe it briefly. Let $(f,\pi)\in G\wr S_n$. Let $(j,\pi(j),\ldots, \pi^t(j))$ be a cycle appearing in $\pi$ (when written as a product of disjoint cycles). Define the cycle product corresponding to that cycle in $\pi$, as the unique element in $G$, given by,
	$$f(j)f_{\pi}(j)\ldots f_{\pi^t}(j)=f(j)f(\pi^{-1}(j))\ldots f(\pi^{-t}(j)).$$  
	
	\noindent Let $\C_1, \C_2, \ldots, \C_s$ denote a labeling of conjugacy classes of $G$. We define the type of an element $(f,\pi)\in G\wr S_n$ as the $s\times n$ matrix $(a_{ik})$ where, $a_{ik}$ is the number of $k$-cycles in $\pi$, whose cycle product belongs to $\C_i$. If we denote type($\pi$)=$(c_1,c_2,\ldots)$, then it is clearly seen that $a_{ik}\in \mathbb{Z}_{\geq 0}$ for every $1\leq i\leq s,1 \leq k \leq n$ and $\sum\limits_{i}ka_{ik}=c_k$ for every $k=1,2,\ldots, n$. Moreover, $\sum\limits_{i,k}ka_{ik}=n$. The following result which we state without proof determines the conjugacy classes of $G\wr S_n$, and associates a combinatorial data to each such conjugacy class.
	
	\begin{proposition}[3.7, \cite{ke}]\label{conjinwr}
		Let $(f,\pi), (g,\pi^{'}) \in G\wr S_n$. Then $(f,\pi)$ is conjugate to $(g,\pi^{'})$, if and only if $\mathrm{type}(f,\pi)=\mathrm{type}(g,\pi^{'})$. Therefore, the conjugacy classes in $G\wr S_n$ are in one-one correspondence with $s\times n$ matrices $(a_{ik})$, with $a_{ik}\in \mathbb{Z}_{\geq 0}$ for every $i,k$ and $\sum\limits_{i,k}ka_{ik}=n$. 
	\end{proposition}
	
	\noindent The type of $(f,\pi)$ is well-defined in the sense, that the cycle product doesn't depend on the representation of a cycle $(j,\pi(j),\ldots, \pi^t(j))$ in $\pi$. In other words, if $(j,\pi(j),\ldots, \pi^t(j))=(l,\pi(l),\ldots, \pi^t(l))$, then, their corresponding cycle products $f(j)f_{\pi}(j)\ldots f_{\pi^t}(j)$ and $f(l)f_{\pi}(l)\ldots f_{\pi^t}(l)$ are conjugate in $G$. Later we will see that, in fact this gives us a degree of freedom, to choose a suitable cycle product in various computations.
	
	Given $g=(f,\pi)\in G\wr S_n$, let $\mathcal{Z}(g)$ denote the centralizer of $g$ in $G\wr S_n$. The structure of the centralizer is well-known. We will just need to know the order of the centralizer for our purpose.
	
	\begin{proposition}[3.9, \cite{ke}]\label{sizecentra}
		Let $g=(f,\pi)\in G\wr S_n$ and, $\mathrm{type}(g)=(a_{ik})$ which is a $s\times n$ matrix with $a_{ik}\in \mathbb{Z}_{\geq 0}$ and, $\sum\limits_{i,k}ka_{ik}=n$. Then, $|\mathcal{Z}(g)|=\prod\limits_{i,k}a_{ik}!(k|G|/|\C_i|)^{a_{ik}}$.
	\end{proposition}

	\noindent With these informations we can now define the cycle index for $G\wr S_n$. For convenience, we will denote the type of $g\in G\wr S_n$, as $T(g)=(T(g)_{ij})$, which is a $s\times n$ matrix. We define the cycle index of $G\wr S_n$ as,
	
	\begin{equation}
	Z_n(t_{11},t_{12},\ldots,t_{1n},\ldots,t_{s1},t_{s2},\ldots, t_{sn}; G\wr S_n)=\frac{1}{|G|^nn!}\sum\limits_{g\in G} \left( \prod\limits_{i,j} t_{ij}^{T(g)_{ij}} \right).
	\end{equation}
	
	\noindent The cycle index of $G\wr S_n$ is a polynomial in the variables $t_{ij}$ for $1\leq i\leq s, 1\leq j\leq n$. We call this polynomial the cycle polynomial associated with $G\wr S_n$. Observe that given a $s\times n$ matrix $(c_{ij})$, with $\sum_{i,j}jc_{ij}=n$ (that is, it is a type of an element in $G\wr S_n$), the coefficient of the monomial $\prod\limits_{i,j}t_{ij}^{c_{ij}}$, is equal to $\frac{|Cl(g)|}{|G|^nn!}$, where $g$ is such that $T(g)=(c_{ij})$, and $Cl(g)$ denotes the conjugacy class of $g$ in $G\wr S_n$. Thus, the coefficient of a monomial in $Z_n$, is precisely the probability that a randomly chosen element of $G\wr S_n$ belongs to that conjugacy class. The cycle-index generating function of $G\wr S_n$ is,
	\begin{equation}
		1+\sum_{n=1}^{\infty}Z_nu^n=1+\sum_{n=1}^{\infty}\frac{1}{|G|^nn!}\left[\sum\limits_{g\in G} \left( \prod\limits_{i,j} t_{ij}^{T(g)_{ij}} \right)\right]u^n.
	\end{equation}
	
	\noindent We see some examples of cycle index in $G\wr S_n$.
	\begin{example}\label{ex1}
	Let $G=C_2=\{\pm 1\}$, the cyclic group of order 2. Then $C_2\wr S_2$ is a group of order 8, which turns out to be the Dihedral group $D_8$ of order 8. Let $\C_1=\{1\}, \C_2=\{-1\}$ be the two conjugacy classes of $C_2$. The conjugacy classes of $C_2\wr S_2$ are parametrized by $2\times 2$ matrices $(a_{ik})$, with $a_{ik}\in \mathbb{Z}_{\geq 0}$, and $\sum\limits_{i,k}ka_{ik}=2$. These matrices are,
	\[
	\begin{bmatrix}
		2 & 0 \\
		0 & 0
	\end{bmatrix}
	, \begin{bmatrix}
	0 & 0\\
	2 & 0
	\end{bmatrix}
	, \begin{bmatrix}
	1 & 0\\
	1 & 0
	\end{bmatrix}
	, \begin{bmatrix}
	0 & 1\\
	0 & 0
	\end{bmatrix}
	, \begin{bmatrix}
	0 & 0\\
	0 & 1
	\end{bmatrix}.
	\]
	
	\noindent Thus there are exactly 5 conjugacy classes in $C_2\wr S_2$. Observe that the coefficient of each monomial, in the cycle index as discussed above is the number of elements in the conjugacy class represented by that monomial divided by the size of the group. Consider, for example the monomial $t_{11}t_{21}$. This monomial represents the conjugacy class parametrized by \[\begin{bmatrix}
	1 & 0\\ 1 & 0 \end{bmatrix} \]
	whose conjugacy class size is $\frac{8}{4}=2$, using Proposition~\ref{sizecentra}, by putting $|G|=2, |\C_1|=1, |\C_2|=1$. By calculating this coefficient for each monomial, we get the cycle index for $C_2\wr S_2$ as follows
	
	\begin{equation}
	Z_2(t_{11}, t_{12}, t_{21}, t_{22}; C_2\wr S_2)=\frac{1}{8}(t_{11}^2+t_{21}^2+2t_{11}t_{21}+2t_{12}+2t_{22}).
	\end{equation}
	\end{example}
	
	\begin{example}\label{ex2}
		We take a look at another example. Consider, $S_3\wr S_2$ which is a group of order 72. Let $\C_1=\{e\}, \C_2=\{(1 2), (2 3), (1 3)\}, \C_3=\{(1 2 3), (1 3 2)\}$ be the three conjugacy classes in $S_3$. The conjugacy classes in $S_3\wr S_2$ are given by $3\times 2$ matrices $(a_{ik})$, where $a_{ik}\in \mathbb{Z}_{\geq 0}$, and $\sum\limits_{i,k}ka_{ik}=2$. These matrices are,
			\[
		\begin{bmatrix}
		2 & 0 \\
		0 & 0\\
		0 & 0
		\end{bmatrix}
		, \begin{bmatrix}
		0 & 0\\
		2 & 0\\
		0 & 0\\
		\end{bmatrix}
		, \begin{bmatrix}
		0 & 0\\
		0 & 0\\
		2 & 0
		\end{bmatrix}
		, \begin{bmatrix}
		1 & 0\\
		1 & 0\\
		0 & 0
		\end{bmatrix}
		, \begin{bmatrix}
		0 & 0\\
		1 & 0\\
		1 & 0
		\end{bmatrix}
		, \begin{bmatrix}
			1 & 0\\
			0 & 0\\
			1 & 0
		\end{bmatrix}
		, \begin{bmatrix}
			0 & 1\\
			0 & 0\\
			0 & 0
		\end{bmatrix}
		, \begin{bmatrix}
		0 & 0\\
		0 & 1\\
		0 & 0
		\end{bmatrix}
		, \begin{bmatrix}
		0 & 0\\
		0 & 0\\
		0 & 1
		\end{bmatrix}.
		\]
		Thus there are nine conjugacy classes. Once again by calculating the coefficient of each monomial using the formula for centralizers, the cycle polynomial of $S_3\wr S_2$ is given by, 
		\begin{equation}
				\frac{1}{72}(t_{11}^2+9t_{21}^2+ 4t_{31}^2+6t_{11}t_{21}+12t_{11}t_{31}+4t_{21}t_{31}+6t_{12}
			+18t_{22}+12t_{32}).
		\end{equation}
	\end{example}

	The following proposition expresses the cycle index generating function of $G\wr S_n$ in an infinite product form, which will be useful later.
	
	\begin{proposition}\label{cycleindex}
		Let $G$ be a finite group. Let $Z_n$ denote the cycle index of $G\wr S_n$. Let $\C_1,\C_2,\ldots, \C_s$ denote a labeling of conjugacy classes of $G$. We have,
		$$1+\sum\limits_{n=1}^{\infty}Z_nu^n=\prod_{i=1}^{s}\prod_{j=1}^{\infty}\mathrm{exp}\left(t_{ij}\frac{|\C_i|u^j}{|G|j}\right)$$
	\end{proposition} 

	\begin{proof}
		Observe first that coefficient of a monomial $\prod\limits_{i,j}t_{ij}^{a_{ij}}$ is one over the size of the centralizer of the conjugacy class in $G\wr S_n$, parametrized by the $s\times n$ matrix $(a_{ij})$. By Proposition~\ref{sizecentra}, this is 
		$$\frac{1}{\prod\limits_{i,j}a_{ij}!(j|G|/|\C_i|)^{a_{ij}}}=\frac{|\C_i|^{a_{ij}}}{a_{ij}!j^{a_{ij}}|G|^{a_{ij}}}$$
		Using this it is easy to see that the coefficient of $u^{n}$ in $\prod\limits_{i=1}^{s}\prod\limits_{j=1}^{\infty}\mathrm{exp}\left(t_{ij}\frac{|\C_i|u^j}{|G|j}\right)$ precisely yields $Z_n$, which is the cycle polynomial of $G\wr S_n$.
	\end{proof}
	 \noindent Observe that taking $G=\{1\}$, the above proposition is the factorization (in infinite product form) for the cycle index generating function of $S_n$ (see Equation~\ref{cycleindex-genfunction-S_n}, Section 2).

	\section{Computing powers in $G\wr S_n$}
	
	In this section, we compute the $r^{th}$ powers in $G\wr S_n$, where $r\geq 2$ is prime. We begin with some preparatory lemmas. The proof of the following lemma is easy.
	
	\begin{lemma}\label{preplem1}
		Let $M\geq 2$ be any positive integer. Suppose $G$ is a finite group. Then, the power map $\omega_M:G\to G$ defined by $g\mapsto g^M$ is surjective if and only if $(M,|G|)=1$
	\end{lemma}
	
	%\begin{proof}
		%Easy.
		%Let us assume the map is surjective. Then, it is injective. Thus, $g^M=e \implies g=e$. This clearly implies $(M,|G|)=1$. Conversely, assume $(M,|G|)=1$. Let $|G|=k$. Thus, there exists integers $a,b$ such that $Ma+kb=1$. Let us define the map, $\theta_M:G\to G$ given by $g\mapsto g^a$. Then, $\theta_M$ is the inverse of $\omega_M$ for, $$\omega_M\circ \theta_M(g)=\omega_M(g^a)=g^{Ma}=g^{1-kb}=g.$$
		%Thus, $\omega_M$ is bijective, hence surjective.
	%\end{proof}
	
	\noindent The following lemma is the most important for our purpose. We set an important convention at this point. Given a $s\times n$ matrix $(a_{ij})$, we assume $a_{ij}=0$ for all $i>s$ or, $j>n$. We further assume for $1\leq i\leq s$,
	\begin{center}
		$(\C_i)^r=\{x^r \mid  x\in \C_i \}$
	\end{center}
	where $\C_1,\C_2,\ldots,\C_s$ denotes a labeling of the conjugacy classes of $G$. Observe that $(\C_i)^r=\C_j$ for some $1\leq j\leq s$.
	
	\begin{lemma}\label{preplem3}
		Let $g=(f,\pi) \in G\wr S_n$. Let $T(g)=(T(g)_{ij})_{s\times n}$  be  the type of $g$. Then, the type of $g^r$ is given by,
		\[
		T(g^r)_{i,j}=
		\begin{cases}
			\sum\limits_{z=1}^{m}T(g)_{y_z,j}+rT(g)_{i,rj} & \text{ when } r \nmid j\\
			
			rT(g)_{i,rj} & \text{ when } r \mid j
			 
		\end{cases}
		\]
	where, $\{y_1,y_2,\ldots, y_m\}\subset \{1,2,\ldots, s\}$ is the complete set of indices $k$ such that $(\C_k)^r=\C_i$.
	\end{lemma}

	\begin{proof}
		Let $g=(f,\pi) \in G\wr S_n$, and $T(g)=(T(g)_{ij})_{s\times n}$ denotes the type of $g$. Let the type of $\pi$ be given by $(c_1,c_2,\ldots, c_n)$. Observe that,
		$$g^r=(f,\pi)^r=(ff_{\pi}\ldots f_{\pi^{r-1}},\pi^r).$$
		Let, $h=ff_{\pi}\ldots f_{\pi^{r-1}} \in G$. To find $T(g^r)$, which is the type of $g^r$, we need to compute the cycle products corresponding to each cycle in $\pi^r$. Let us fix $1\leq i\leq s, 1\leq j\leq n$. By Lemma~\ref{power-in-Sn}, we have to consider two cases. For the first case, let us assume that $r\mid j$. Then, we know by Lemma~\ref{power-in-Sn}, if $(d_1,d_2,\ldots, d_n)$ be the type of $\pi^r$, then $d_j=rc_{rj}$. In other words, a $j$-cycle in $\pi^r$, can only be obtained by raising a $rj$-cycle to the power of $r$. Let us consider a $rj$-cycle in $\pi$ of the form,
		$$\sigma=(l_1,l_2,\ldots, l_r, l_{r+1}, \ldots, l_{2r}, \ldots, l_{(j-1)r+1}, \ldots, l_{rj})$$ such that the cycle product of $\sigma$ with respect to $f$, which is $$x=f(l_1)f(\sigma^{-1}(l_1))\ldots f(\sigma^{-(rj-1)}(l_1)) \in \C_i.$$
		\noindent Now,
		$$\sigma^r=(l_1,l_{r+1}, \ldots, l_{(j-1)r+1})(l_2,l_{r+2}, \ldots, l_{(j-1)r+2})\ldots (l_r,l_{2r}, \ldots, l_{jr}).$$
		Thus, $\sigma^r=\sigma_1\sigma_2\ldots \sigma_r$, where,
		$\sigma_t=(l_t,l_{r+t}, \ldots, l_{(j-1)r+t})$. For $1\leq t\leq r$, let us find the cycle product of $\sigma_t$ with respect to $h$, which is,
		
		\begin{eqnarray*}
		y &=& h(l_t)h(\sigma_t^{-1}(l_t))\ldots h(\sigma_t^{-(r-1)}(l_t)) \\ &=& f(l_t)f(\pi^{-1}(l_t))\ldots  f(\pi^{-(r-1)}(l_t))f(\sigma_t^{-1}(l_t))  f(\pi^{-1}(\sigma_t^{-1}(l_t)))  \ldots \\ &&  f(\pi^{-(r-1)}(\sigma_t^{-1}(l_t)))  \ldots  f(\sigma_t^{-(r-1)}(l_t))f(\pi^{-1}(\sigma_t^{-(r-1)}(l_t))) \ldots \\ && f(\pi^{-(r-1)}(\sigma_t^{-(r-1)}(l_t))) \\ &=&  f(l_t)f(\sigma^{-1}(l_t))\ldots f(\sigma^{-(r-1)}(l_t))f(\sigma^{-r}(l_t))  f(\sigma^{-(r+1)}(l_t))\ldots f(\sigma^{-(2r-1)}(l_t))\ldots \\ && f(\sigma^{-(j-1)r}(l_t))\ldots f(\sigma^{-(rj-1)}(l_t)).
		\end{eqnarray*}
		
		\noindent It is clear that $x$ is conjugate to $y$ in $G$, and hence, $y\in \C_i$. We know, $T(g)_{i,rj}$ denotes the number of $rj$-cycles in $\pi$, whose cycle product belongs to $\C_i$. One $rj$-cycle in $\pi$ contributes $r$ $j$-cycles in $\pi^r$, and by the above computation the cycle product of each such $j$-cycle belongs to $\C_i$, if the cycle product of the former belongs to $\C_i$. We therefore conclude that in the case when $r\mid j$, $T(g^r)_{i,j}=rT(g)_{i,rj}$. 
		
		\noindent Let us now move on to the second case, in which $r\nmid j$. Once again, by Lemma~\ref{power-in-Sn}, 
		$$d_j=c_j+rc_{rj}.$$ 
		Thus, a $j$-cycle in $\pi^r$ is obtained from $rj$-cycles in $\pi$ just as above, as well as, from a $j$-cycle in $\pi$. We need to address the case of the latter possibility, since, the former one has already been taken care of in the previous case. Consider a $j$-cycle in $\pi$ say, $\sigma=(l_1,l_2,\ldots, l_j)$. The cycle product of $\sigma$ with respect to $f$ is $x=f(l_1)f(\sigma^{-1}(l_1))\ldots f(\sigma^{-(j-1)}(l_1))$. Since $\sigma^j=id$, the cycle product of $\sigma^r$ with respect to $h$ is
		\begin{eqnarray*}
				y&=&h(l_1)h(\sigma^{-r}(l_1))\ldots h(\sigma^{-r(j-1)}(l_1))=f(l_1)f(\pi^{-1}(l_1))\ldots  f(\pi^{-(r-1)}(l_1)) \\ && f(\sigma^{-r}(l_1)) f(\pi^{-1}(\sigma^{-r}(l_1)))\ldots f(\pi^{-(r-1)}(\sigma^{-r}(l_1)))\ldots  f(\sigma^{-r(j-1)}(l_1)) \\ && f(\pi^{-1}(\sigma^{-r(j-1)}(l_1))) \ldots f(\pi^{-(r-1)}(\sigma^{-r(j-1)}(l_1))) \\ &=& f(l_1)f(\sigma^{-1}(l_1))\ldots  f(\sigma^{-(r-1)}(l_1))f(\sigma^{-r}(l_1))f(\sigma^{-(r+1)}(l_1))\ldots f(\sigma^{-(2r-1)}(l_1)) \\ && \ldots f(\sigma^{-r(j-1)}(l_1))  f(\sigma^{-(r(j-1)+1)}(l_1))\ldots f(\sigma^{-(r(j-1)+(r-1))}(l_1)) \\ &=& f(l_1)f(\sigma^{-1}(l_1))\ldots f(\sigma^{-(j-1)}(l_1))f(\sigma^{-j}(l_1))f(\sigma^{-(j+1)}(l_1))\ldots f(\sigma^{-(2j-1)}(l_1)) \\ && \ldots f(\sigma^{-(r-1)j}(l_1))\ldots  f(\sigma^{-((r-1)j+j-1)}(l_1)) \\ &=& f(l_1)f(\sigma^{-1}(l_1))\ldots f(\sigma^{-(j-1)}(l_1))f(l_1)f(\sigma^{-1}(l_1))\ldots f(\sigma^{-(j-1)}(l_1))\ldots \\ && f(l_1)f(\sigma^{-1}(l_1))\ldots   f(\sigma^{-(j-1)}(l_1))=\underbrace{x.x\ldots x}_{r \text{ times}}=x^r.
		\end{eqnarray*}
	
		\noindent Thus, we see that the cycle product $y\in (\C_i)^r$. Since for $1\leq k\leq s$, $T(g)_{k,j}$ denotes the number of $j$-cycles each of  whose cycle product belongs to $\C_k$, we can conclude that if $\{y_1,y_2,\ldots ,y_m\}\subset \{1,2,\ldots, s\}$ be the complete set of indices $k$, such that $(\C_k)^r=\C_i$, then $T(g^r)_{i,j}$ has a contribution of $\sum\limits_{z=1}^{m}T(g)_{y_z,j}$ coming as cycle products of $j$-cycles in $\pi^r$. Therefore, we can finally conclude that, when $r\nmid j$
		$$T(g^r)_{i,j}=\sum\limits_{z=1}^{m}T(g)_{y_z,j}+ rT(g)_{i,rj}.$$
		This completes the proof.
	\end{proof}
	
	\noindent We can now characterize the conjugacy classes in $G\wr S_n$ that are $r^{th}$ powers entirely in terms of the associated combinatorial data.
	
	\begin{proposition}\label{power-in-wreath}
		Let $G$ be a finite group which has $s$ number of conjugacy classes. Let out of $s$, exactly $d$ such conjugacy classes are not $r^{th}$ powers in $G$.  Let $\C_1, \C_2, \ldots , \C_s$ be a labeling of conjugacy classes of $G$ such that $\C_1,\C_2, \ldots, \C_d$ are not $r^{th}$ powers. Let $g\in G\wr S_n$, and $T(g)_{s\times n}=(T(g)_{ij})$ denotes the type of $g$. Then, $g$ is a $r^{th}$ power in $G\wr S_n$ if and only if $r \mid T(g)_{ij}$ whenever $r\mid j$ or, $1\leq i\leq d$.
	\end{proposition}

	\begin{proof}
		Assume that $g \in G\wr S_n$ is a $r^{th}$ power, that is, there exists $h\in G\wr S_n$ such that $h^n=g$. Let $T(h)$ denotes the type of $h$. The type of $h^n$ is given by Lemma~\ref{preplem3}. Since $T(h^n)=T(g)$, we at once conclude that $r\mid T(g)_{i,j}$, whenever $r\mid j$. Since the conjugacy classes of $G$ are labeled as $\C_1,\C_2,\ldots,\C_s$ such that $\C_1,\C_2,\ldots, \C_d$ are not $r^{th}$ powers, we conclude once again by Lemma~\ref{preplem3}, that, for $1\leq i\leq d$, $T(g)_{i,j}=rT(h)_{i,rj}$, thus showing that $r\mid T(g)_{i,j}$. Conversely let us assume that $g\in G\wr S_n$ is such that $r|T(g)_{ij}$ whenever $1\leq i\leq d$ or, $r\mid j$. We need to show that there exists $h\in G\wr S_n$ such that $h^n=g$. For $d+1\leq i\leq s$, let $\C_{i_1}, \C_{i_2},\ldots \C_{i_u}$ be the complete set of classes such that $\C_{i_m}^r=\C_{i}$ for all $1\leq m\leq u$. Moreover we choose the indices $i_m$ in increasing order, that is, $i_1<i_2<\ldots <i_u$. It is clear that for two distinct $e,f \in \{d+1,\ldots, s \}$, the indices $e_1,e_2,\ldots, e_u$ and the indices $f_1,f_2,\ldots, f_{u^{'}}$ are disjoint. Consider the $s\times n$ matrix $T^{'}$ defined as follows:
		
		\noindent When $r\nmid j$,
		\[T^{'}_{ij}=
		\begin{cases}
		T(g)_{i_1,j} & ;\text{ if } d+1\leq i\leq s\\
		0 & ;\text{otherwise}
		\end{cases}
		\]
		and, 
		\[
		T^{'}_{i,rj}=
		\begin{cases}
		\frac{T(g)_{i,j}}{r} & ;\text{ if } 1\leq i\leq d \\
		\frac{T(g)_{i,j}}{r} & ;\text{ if } d+1\leq i\leq s \text{ and, } r\mid j\\
		0 & ; \text{ if } d+1\leq i\leq s \text{ and, } r\nmid j\\
		\end{cases}
		\]

		\noindent In the above definition, we once again use the convention that if the indices $i,j$ exceeds the bound $1\leq i\leq s, 1\leq j\leq n$, we don't consider those entries. Observe that $\sum\limits_{i,j} jT^{'}_{ij}=n$. Indeed,
		
		\begin{equation}\label{sum-check}
		\displaystyle \sum_{i,j}jT^{'}_{ij}=\sum_{i,r\mid j}rjT^{'}_{i,rj}+\sum_{i,r\nmid j}jT^{'}_{i,j}+\sum_{i,r\nmid j} rjT^{'}_{i,rj}
		\end{equation}
		
		\noindent Now, it is clear from the definition of $T^{'}$ that,
		$$\displaystyle \sum_{i,r\mid j}rjT^{'}_{i,rj}=\sum_{i,r\mid j}jT(g)_{i,j}$$
		and,
		\begin{eqnarray*}
		\displaystyle \sum_{i,r\nmid j}jT^{'}_{i,j}+\sum_{i,r\nmid j} rjT^{'}_{i,rj} &=& \sum_{i,r\nmid j}jT^{'}_{i,j}+\sum_{\substack{1\leq i\leq d\\r\nmid j}} rjT^{'}_{i,rj} = \sum_{i,r\nmid j}jT^{'}_{i,j}+\sum_{\substack{1\leq i\leq d \\ r\nmid j}}jT(g)_{i,j} \\ &=& \sum_{\substack{d+1\leq i\leq s \\ r\nmid j}}jT(g)_{i,j}+\sum_{\substack{1\leq i\leq d \\ r\nmid j}}jT(g)_{i,j} =\sum_{i,r\nmid j}jT(g)_{i,j}.
		\end{eqnarray*}
		
		\noindent Thus from Equation~\ref{sum-check}, we get,
		$$\displaystyle \sum_{i,j}jT^{'}_{ij}=\sum_{i,r\mid j}jT(g)_{ij}+\sum_{i,r\nmid j}jT(g)_{ij}=\sum_{i,j}jT(g)_{ij}=n.$$
		Therefore, there exists $h\in G\wr S_n$ such that $T(h)=T^{'}$, where $T(h)$ is the type of $h$. Now, by Lemma~\ref{preplem3}, when $r\mid j$, $T(h^n)_{ij}=rT(h)_{i,rj}=T(g)_{i,j}.$
		When $r\nmid j$, and $1\leq i\leq d$,
		$T(h^n)_{ij}=rT(h)_{i,rj}=T(g)_{ij}.$
		When $r\nmid j$ and $d+1\leq i\leq s$, once again by Lemma~\ref{preplem3}, and the fact that $i_1,\ldots,i_u$ are the complete set of indices such that $\C_{i_m}^r=\C_i$,
		$$\displaystyle T(h^n)_{ij}=\sum_{m=1}^{u} T(h)_{i_m,j}+rT(h)_{i,rj}=T(g)_{ij}$$
		Thus, $T(h^n)=T(g)$ which implies that $h^n$ is conjugate to $g$. Therefore there exists $x\in G\wr S_n$ such that $xh^nx^{-1}=g \implies (xhx^{-1})^n=g$. This completes the proof.
	\end{proof}
	
	\begin{corollary}\label{corr1}
		Let $G$ be a finite group with $r\nmid |G|$. Let $g\in G\wr S_n$, and $T(g)_{s\times n}=(T(g)_{ij})$ denotes the $\mathrm{type}$ of $g$. Then, $g$ is a $r^{th}$ power in $G\wr S_n$ if and only if $r\mid T(g)_{ij}$, whenever $r\mid j$.
	\end{corollary}

	\begin{proof}
		Since by Lemma~\ref{preplem1}, all conjugacy classes in $G$ are $r^{th}$ powers, we get the corollary by putting $d=0$ in Proposition~\ref{power-in-wreath} above. 
	\end{proof}
	
	We end this section with some examples.
	
	\begin{example}\label{ex3}
	 Let us consider $G=C_3$ which is the cyclic group of order 3. Consider $C_3\wr S_3$ which is a group of order 162. By Proposition~\ref{conjinwr} of previous section, conjugacy classes in $C_3\wr S_3$ are in one-one correspondence with $3\times 3$ matrices $(a_{ij})$ with non-negative integer entries and satisfying $\sum\limits_{i,j}ja_{ij}=3$. There are 22 conjugacy classes in $C_3\wr S_3$. Let us list down the 22 matrices, that parametrize the conjugacy classes in $C_3\wr S_3$. 
	 \[x_1=
	 \begin{pmatrix}
	 3 & 0 & 0\\0 & 0 & 0\\0 & 0 & 0
	 \end{pmatrix}, x_2=
	 \begin{pmatrix}
	 0 & 0 & 0\\3 & 0 & 0\\0 & 0 & 0
	 \end{pmatrix},
	 x_3=
	 \begin{pmatrix}
	 0 & 0 & 0\\0 & 0 & 0\\3 & 0 & 0
	 \end{pmatrix},
	 x_4=
	 \begin{pmatrix}
	 2 & 0 & 0\\1 & 0 & 0\\0 & 0 & 0
	 \end{pmatrix},
	 \]
	 
	\[ x_5=
	\begin{pmatrix}
	0 & 0 & 0\\2 & 0 & 0\\1 & 0 & 0
	\end{pmatrix},
	x_6=
	\begin{pmatrix}
	2 & 0 & 0\\0 & 0 & 0\\1 & 0 & 0
	\end{pmatrix},
	x_7=
	\begin{pmatrix}
	1 & 0 & 0\\2 & 0 & 0\\0 & 0 & 0
	\end{pmatrix},
	x_8=\begin{pmatrix}
	0 & 0 & 0\\1 & 0 & 0\\2 & 0 & 0
	\end{pmatrix},
	\]
	
	\[ x_9=
	\begin{pmatrix}
	1 & 0 & 0\\0 & 0 & 0\\2 & 0 & 0
	\end{pmatrix},
	x_{10}=
	\begin{pmatrix}
	1 & 0 & 0\\1 & 0 & 0\\1 & 0 & 0
	\end{pmatrix},
	x_{11}=
	\begin{pmatrix}
	1 & 1 & 0\\0 & 0 & 0\\0 & 0 & 0
	\end{pmatrix},
	x_{12}=\begin{pmatrix}
	1 & 0 & 0\\0 & 1 & 0\\0 & 0 & 0
	\end{pmatrix},
	\]
	
	\[ x_{13}=
	\begin{pmatrix}
	1 & 0 & 0\\0 & 0 & 0\\0 & 1 & 0
	\end{pmatrix},
	x_{14}=
	\begin{pmatrix}
	0 & 1 & 0\\1 & 0 & 0\\0 & 0 & 0
	\end{pmatrix},
	x_{15}=
	\begin{pmatrix}
	0 & 0 & 0\\1 & 1 & 0\\0 & 0 & 0
	\end{pmatrix},
	x_{16}=\begin{pmatrix}
	0 & 0 & 0\\1 & 0 & 0\\0 & 1 & 0
	\end{pmatrix},
	\]
	
	\[ x_{17}=
	\begin{pmatrix}
	0 & 1 & 0\\0 & 0 & 0\\1 & 0 & 0
	\end{pmatrix},
	x_{18}=
	\begin{pmatrix}
	0 & 0 & 0\\0 & 1 & 0\\1 & 0 & 0
	\end{pmatrix},
	x_{19}=
	\begin{pmatrix}
	0 & 0 & 0\\0 & 0 & 0\\1 & 1 & 0
	\end{pmatrix},
	x_{20}=\begin{pmatrix}
	0 & 0 & 1\\0 & 0 & 0\\0 & 0 & 0
	\end{pmatrix},
	\]
	\[ x_{21}=
	\begin{pmatrix}
	0 & 0 & 0\\0 & 0 & 1\\0 & 0 & 0
	\end{pmatrix} \text{ and, }
	x_{22}=
	\begin{pmatrix}
	0 & 0 & 0\\0 & 0 & 0\\0 & 0 & 1
	\end{pmatrix}.
	\]
	
	\noindent With this data, let us compute which conjugacy classes are squares (that is, $r=2$). Observe that $|C_3|=3$, which is odd, thus by Corollary~\ref{corr1}, we conclude that the conjugacy classes which are squares are parametrized by the matrices $(a_{ij})_{1\leq i,j\leq 3}$ with $a_{ij}$ being even whenever $j$ is even. Thus, conjugacy classes from $x_1$ to $x_{10}$, and $x_{20}$ to $x_{22}$ are square conjugacy classes. Thus there are 13 conjugacy classes out of 22 conjugacy classes in $C_3\wr S_3$, which are squares. A direct computation shows that exactly 81 elements are squares in $C_3\wr S_3$.
	\end{example}

	\begin{example}\label{ex4}
		Consider the wreath product $S_3\wr S_3$ which is a group of order 1296. Let $\C_1=\{(1 2), (2 3), (1 3)\},\text{ } \C_2=\{e\},\text{ } \C_3=\{(1 2 3), (1 3 2)\}$ be a labelling of the conjugacy classes in $S_3$. Observe that $\C_1$ is the only non-square conjugacy class in $S_3$. Observe again by using Proposition~\ref{conjinwr}, the conjugacy classes in $S_3\wr S_3$ are in one-one correspondence with $3\times 3$ matrices $(a_{ij})$, with non-negative integer entries and $\sum\limits_{i,j}ja_{ij}=3$. This is the same set of matrices as described in the previous example in case $C_3\wr S_3$. Thus, the matrices $x_1$ to $x_{22}$ gives the conjugacy class types of $S_3\wr S_3$ once again. We compute the squares (that is, $r=2$) in $S_3\wr S_3$. Using Proposition~\ref{power-in-wreath}, we conclude that the conjugacy classes in  $S_3\wr S_3$ that are squares, are parametrized by the $3\times 3$ matrices $(a_{ij})$, with the property that $a_{ij}$ is even whenever, either $j=2$, or $i=1$. We conclude that conjugacy classes parametrized by matrices $x_2,x_3,x_4,x_5,x_6,x_8,x_{21},x_{22}$ satisfy the required properties, and hence are square conjugacy classes. Therefore, in $S_3\
		\wr S_3$, 8 conjugacy classes are squares out of possible 22 classes. Once again a direct computation shows that 324 elements out of the total 1296 are squares.
	\end{example}

	\section{Generating functions for the powers in $G\wr S_n$}
	
	In this section, we prove one of our main results, using generating functions. Recall that for the group $G\wr S_n$, we have $P_r(G\wr S_n)=\frac{|\omega_r(G\wr S_n)|}{|G|^nn!}$, where $\omega_r(G\wr S_n)$ is the set of all $r^{th}$ powers in $G\wr S_n$. Let $\alpha_i=\frac{|\C_i|}{|G|}$ for $1\leq i\leq s$, where $\C_1, \C_2, \ldots, \C_s$ denotes a labeling of conjugacy classes of $G$. For $j\geq 1$, define, the function $\varphi_j(u)$ as 
	$$\varphi_j(u)=\exp{\frac{u^{j}}{j}}.$$
	Let $\omega\neq 1$ be a $r^{th}$ root of unity. We define, for $j\geq 1$, and $1\leq i\leq s$,
	
	$$\psi_{j,\alpha_i}(u)=\frac{1}{r}\sum\limits_{k=1}^{r-1}\varphi_j(u)^{\alpha_i\omega^k}.$$
	
	\noindent The following proposition gives the generating function for $P_r(G\wr S_n)$.
	\begin{proposition}\label{mainprobgen}
		Let $G$ be a finite group, and $r\geq 2$ be a prime. Let $\C_1, \C_2, \ldots ,\C_s$ be a labeling of conjugacy classes of $G$ in such a way that $\C_1, \C_2, \ldots, \C_d$ are those conjugacy classes which are not $r^{th}$ powers in $G$ for some $0\leq d<s$. Then,
		
		$$1+\sum\limits_{n=1}^{\infty} P_r(G\wr S_n)u^n=\prod\limits_{i=1}^{d}\left[\prod\limits_{j=1}^{\infty} \psi_{j,\alpha_i}(u)\right]\prod_{i=d+1}^{s}\left[\left(\frac{(1-u^r)^{1/r}}{(1-u)}\right)^{\alpha_i} \prod_{j=1}^{\infty}\psi_{rj,\alpha_i}(u)\right].$$
	\end{proposition}
	
	\begin{proof}
	 We will make use of the cycle index generating function of $G\wr S_n$ in Proposition~\ref{cycleindex}, and the characterization of $r^{th}$ powers given by Proposition~\ref{power-in-wreath}. By Proposition~\ref{cycleindex}, we have,
	 \begin{equation}\label{cycleindexeq}
	 1+\sum\limits_{n=1}^{\infty}Z_nu^n=\prod_{i=1}^{s}\prod_{j=1}^{\infty}\text{exp}\left(t_{ij}\alpha_i\frac{u^j}{j}\right).
	 \end{equation}	
	 
	\noindent By Proposition~\ref{power-in-wreath}, we know that if $g$ is a $r^{th}$ power in $G\wr S_n$, then each entry in the  first $d$ rows of $T(g)$ is divisible by $r$, where $T(g)=\text{type}(g)$. Therefore, to include the conjugacy classes in $G\wr S_n$, which are $r^{th}$ powers in $G\wr S_n$, for $1\leq i\leq d$, $j\geq 1$ we put $t_{ij}=1$ in Equation~\ref{cycleindexeq}. Now only the coefficients of $u^{mr}$ need to be included in the formal power series expansion of   $\text{exp}(\alpha_i\frac{u^j}{j})$ for each $m\geq 1$. This is achieved when we replace $\text{exp}(\alpha_i\frac{u^j}{j})$ by $\psi_{j,\alpha_i}(u)$. For the rows from $d+1$ to $s$ in $T(g)$, the entries in the $j^{th}$ column are multiples of $r$, whenever $r\mid j$. Thus, for $d+1\leq i\leq s$, and for each $j\geq 1$, we put $t_{ij}=1$, whenever $r\nmid j$. Whenever $r\mid j$, we put $t_{ij}=1$, as well as replace $\text{exp}(\alpha_i\frac{u^j}{j})$ by $\psi_{j,\alpha_i}(u)$ in Equation~\ref{cycleindexeq}, by the same argument as in the previous case. Thus, we get, from equation~\ref{cycleindexeq}, plugging in the above mentioned constraints,

\begin{equation}\label{powerindex}
1+\sum\limits_{n=1}^{\infty}P_r(G\wr S_n)u^n=\prod\limits_{i=1}^{d}\left[\prod\limits_{j=1}^{\infty}\psi_{j,\alpha_i}(u)\right]\prod\limits_{i=d+1}^{s}\left[\prod\limits_{\substack{j=1\\ r\nmid j}}^{\infty}\text{exp}\left(\alpha_i\frac{u^j}{j}\right)\right]\left[\prod\limits_{\substack{j=1\\ r\mid j}}^{\infty}\psi_{j,\alpha_i}(u)\right].
\end{equation}

Now, $$\prod\limits_{\substack{j=1 \\ r\nmid j}}^{\infty}\text{exp}\left(\frac{u^j}{j}\right)=\frac{\prod\limits_{j=1}^{\infty}\text{exp}\left(\frac{u^j}{j}\right)}{\prod\limits_{\substack{j=1 \\ r\mid j}}^{\infty}\text{exp}\left(\frac{u^j}{j}\right)}=\frac{1}{(1-u)\prod\limits_{\substack{j=1 \\ r\mid j}}^{\infty}\text{exp}\left(\frac{u^j}{j}\right)}=\frac{1}{(1-u)\prod\limits_{j=1}^{\infty}\text{exp}\left(\frac{u^{rj}}{rj}\right)}$$$$=\frac{1}{(1-u)\text{exp}\left(-\frac{1}{r}\text{log}(1-u^r)\right)}=\frac{(1-u^r)^{\frac{1}{r}}}{1-u}$$
Thus, the final result follows from the above computation and equation~\ref{powerindex}.
	\end{proof}

	\begin{corollary}\label{probgen-coprime}
		Let $r\geq 2$ be a prime, and $G$ be a finite group with $r\nmid |G|$. Assume that $G$ has $s$ number of conjugacy classes. Then,
		$$1+\sum\limits_{n=1}^{\infty} P_r(G\wr S_n)u^n=\left(\frac{(1-u^r)^{1/r}}{(1-u)}\right)\prod\limits_{i=1}^{s}\prod_{j=1}^{\infty}\psi_{rj,\alpha_i}(u).$$
	\end{corollary}
	
	\begin{proof}
	Since $\sum\limits_{i=1}^{s}\alpha_i=1$, the proof follows from the above proposition, where we take $d=0$, due to Corollary~\ref{corr1}.
	\end{proof}
	
	\begin{example}
		Let $G=C_3$ be the cyclic group of order 3. Let $r=2$. By the above corollary,
		\begin{equation}
		\displaystyle 1+\sum_{n=1}^{\infty}P_2(C_3\wr S_n)u^n= \left(\frac{1+u}{1-u}\right)^{\frac{1}{2}}\prod_{j\geq 1} \cosh^3{\left(\frac{u^{2j}}{2j}\right)}
		\end{equation}
		where $\cosh^n(x)=(\cosh{x})^n$.
	\end{example}
	
	\begin{example}
		The Weyl group of type $\mathcal{B}_n$ (or, $\mathcal{C}_n$) for $n\geq 2$, is the wreath product $C_2\wr S_n$. Using Corollary~\ref{probgen-coprime}, the generating function for the proportion of squares in these groups can be given as follows,
		\begin{eqnarray*}
			\displaystyle 1+\sum_{n=1}^{\infty} P_2(\mathcal{B}_n)u^n&=&1+\sum_{n=1}^{\infty} P_2(\mathcal{C}_n)u^n \\ &=&\left(\frac{1+u}{1-u}\right)^{\frac{1}{4}}\prod_{j\geq 1}  \cosh{\left(\frac{u^{j}}{2j}\right)}\cosh{\left(\frac{u^{2j}}{4j}\right)}
		\end{eqnarray*}
		where we assume $\mathcal{B}_1=\mathcal{C}_1=C_2$, and $P_2(\mathcal{B}_n)$ (resp, $P_2(\mathcal{C}_n)$) denotes the proportion of squares in $\mathcal{B}_n$ (resp, $\mathcal{C}_n$). 
	\end{example}
		
	\noindent In particular if we take $G=\{1\}$ and $r=2$ in Corollary~\ref{probgen-coprime}, we get back the generating function for the squares in $S_n$ in Equation~\ref{squares-S_n}. 
	
	Now, we proceed towards proving our main theorem. The following lemma which is a generalization of  Lemma 2 in \cite{bl}, will be needed.
	
	\begin{lemma}\label{mainlemma-genfunction}
		Let $r$ be a prime and $\omega\neq 1$ be a $r^{th}$ root of unity. Let $f(x)$ be a  formal power series such that $f(x)=f(\omega x)$. Suppose, $g(x)=(1+x+\cdots+x^{r-1})f(x)=\sum\limits_{n=0}^{\infty}a_nx^n$. Then, $a_{k+1}=a_{k}$ for  $k \not \equiv -1$ (mod $r$).
	\end{lemma}

	\begin{proof}
		Let $f(x)= \sum\limits_{n=0}^{\infty}b_nx^n$. Since $f(x)=f(\omega x)$, we have, $b_i=0$ for $r\nmid i$. Equating the coefficient of $x^n$  on both sides of $(1+x+\dots+x^{r-1})f(x)=\sum\limits_{n=0}^{\infty}a_nx^n$, we get $a_{k+1}=a_{k}$ for  $k \not \equiv -1$ (mod $r$).                     
	\end{proof}

	\begin{theorem}\label{mainth}
		Let $G$ be a finite group and $r\geq 2$ be a prime, with $(r,|G|)=1$. Then $P_r(G\wr S_{n+1})=P_r(G\wr S_{n})$ for every $n\in \mathbb{N}$, with $n \not\equiv -1(\text{mod } r)$. In other words, $|\omega_r(G\wr S_{n+1})|=|G|(n+1)|\omega_r(G\wr S_n)|$ for every $n\not\equiv -1(\text{mod } r)$.
	\end{theorem}
	\begin{proof}
		We will show that $g(u)=1+\sum\limits_{n=1}^{\infty}P_r(G\wr S_n)u^n$ is such that $g(u)=(1+u+\cdots u^{r-1})f(u)$, for a function $f(u)$ with $f(u)=f(\omega u)$, where $\omega\neq 1$ is a $r^{th}$ root of unity. Then by Lemma ~\ref{mainlemma-genfunction}, the result holds. By Corollary ~\ref{probgen-coprime}, we have,
		
		\begin{equation}\label{maineq}
			g(u)=\left(\frac{(1-u^r)^{1/r}}{(1-u)}\right)\prod\limits_{i=1}^{s}\prod_{j=1}^{\infty}\psi_{rj,\alpha_i}(u).
		\end{equation}
		
		We have,
		\begin{center} $\frac{(1-u^r)^{1/r}}{(1-u)}=\left(\frac{1-u^r}{1-u}\right)(1-u^r)^{\frac{1-r}{r}}=(1+u+\ldots+u^{r-1})(1-u^r)^{\frac{1-r}{r}}$.
		\end{center}
		
	Thus, substituting the above in equation~\ref{maineq}, we get,
	
	$$g(u)=(1+u+\ldots+u^{r-1})(1-u^r)^{\frac{1-r}{r}}\prod_{i=1}^{s}\prod_{j=1}^{\infty}\psi_{rj,\alpha_i}(u).$$
	Take $f(u)=(1-u^r)^{\frac{1-r}{r}}\prod\limits_{i=1}^{s}\prod\limits_{j=1}^{\infty}\psi_{rj,\alpha_i}(u)$.
	Then, $g(u)=(1+u+\cdots +u^{r-1})f(u)$. By the discussion at the beginning of the proof, it remains to prove that $f(u)=f(\omega u)$. Since, $\omega^r=1$, it is clear that $\psi_{j,\alpha_i}(\omega u)=\psi_{j,\alpha_i}(u)$ for all $j\geq 1$ and $1\leq i\leq s$. Finally, 
	$$(1-(\omega u)^r)^{\frac{1-r}{r}}=(1-u^r)^{\frac{1-r}{r}}.$$
	This proves that $f(u)=f(\omega u)$, and the proof follows.
	Finally, by definition, $P_r(G\wr S_{n+1})=P_r(G\wr S_n)$ implies that, $|\omega_r(G\wr S_{n+1})|=|G|(n+1)|\omega_r(G\wr S_n)|$.
	\end{proof}

	\begin{remark}
		If $(r,|G|)=1$, for some finite group $G$, due to Lemma~\ref{preplem1}, the power map $\omega_r:G\to G$ is surjective, and therefore it might naturally seem that $|\omega_r(G\wr S_n)|=|G|^n||\omega_r(S_n)|$. In such a case, the above theorem trivially follows from Theorem~\ref{1}. But we remark that $|\omega_r(G\wr S_n)|=|G|^n|\omega_r(S_n)|$ doesn't hold. As an example, consider $G=C_3$, the cyclic group of order 3. Then the wreath product $C_3\wr S_4$ is a group of order 1944. Using Corollary~\ref{corr1}, taking $r=2$, we see that $|\omega_2(C_3\wr S_4)|=810$. Since $|\omega_2(S_4)|=12$, we have, $|\omega_2(C_3\wr S_4)|=810 \neq 3^4.12=972$.
	\end{remark}
	
	Finally, we give a formula to compute $\mathcal{CC}_r(G\wr S_n)$, which is the number of conjugacy classes in $G\wr S_n$, which are $r^{th}$ powers. Let $\mathcal{CC}(G\wr S_n)$ denote the number of conjugacy classes in $G\wr S_n$. Let the number of conjugacy classes in $G$ be $s$. Then it is known (see 3.8 in \cite{ke}) that,
	
	\begin{equation}\label{numberofconjclass}
	\mathcal{CC}(G\wr S_n)=\sum\limits_{n_1,n_2,\ldots, n_s} p(n_1)p(n_2)\ldots p(n_s)
	\end{equation}
	where $p(t)$ denotes the number of partitions of $t$, and the above sum runs over all ordered $s$-tuples $(n_1,n_2,\ldots ,n_s)$ such that, $n_i\geq 0$ for every $1\leq i\leq s$, and $\sum\limits_{i=1}^{s}n_i=n$.
	The generating function for $\mathcal{CC}(G\wr S_n)$ is easy to write by virtue of the explicit formula for $\mathcal{CC}(G\wr S_n)$. Recall that the generating function for $p(n)$, which is the number of partitions of $n$ is,
	
	\begin{equation}\label{partgen}
	P(u)=1+\sum\limits_{n=1}^{\infty}p(n)u^n=\prod\limits_{i=1}^{\infty}\left(\frac{1}{1-u^i}\right).
	\end{equation}
	
	\noindent By the formula of $\mathcal{CC}(G\wr S_n)$, and using Cauchy product we conclude that,
	\begin{proposition}\label{conjgenfunc}
		$1+\sum\limits_{n=1}^{\infty} \mathcal{CC}(G\wr S_n)u^n=\prod\limits_{i=1}^{\infty} \left(\frac{1}{1-u^i}\right)^s$, where $s$ denotes the number of conjugacy classes of $G$.
	\end{proposition}
	
	 \noindent To write the formula for $\mathcal{CC}_r(G\wr S_n)$, we will need to introduce a notation. Let $p_r(n)$ denote the number of partitions $\lambda=1^{m_1}2^{m_2}\ldots \vdash n$, such that $r \mid m_i$ for each $i\geq 1$. Let $p_{r}^{'}(n)$ denote the number of partitions $\lambda \vdash n$ such that $r \mid m_{ri}$ for each $i\geq 1$. We further assume that $p(0)=p_r(0)=p_r^{'}(0)=1$.
	
	\begin{theorem}\label{powerconjformula}
		Let $r\geq 2$ be a prime and $G$ be a finite group. Suppose, out of the total $s$ number of conjugacy classes of $G$, exactly $d$ of them are not $r^{th}$ powers in $G$. Then,
		\begin{center}
			$\mathcal{CC}_r(G\wr S_n)=\sum\limits_{n_1,n_2,\ldots,n_s}p_r(n_1)\ldots p_r(n_d)p_r^{'}(n_{d+1})\ldots p_r^{'}(n_s)$
		\end{center}
		where, the above sum runs over all ordered $s$-tuples $(n_1,n_2,\ldots ,n_s)$ such that, $n_i\geq 0$ for every $1\leq i\leq s$, and $\sum\limits_{i=1}^{s}n_i=n$.
	\end{theorem}

	\begin{proof}
	 By Proposition~\ref{power-in-wreath}, we can conclude that the conjugacy classes in $G\wr S_n$ that are $r^{th}$ powers are in one-one correspondence, with $s\times n$ matrices $(a_{ij})$, with non-negative integer entries, satisfying, (1) $\sum\limits_{i,j}ja_{ij}=n$, (2) entries in the first $d$ rows are multiple of $r$, (3)entries in the columns which are multiplies of $r$, are divisible by $r$. Since for $1\leq i\leq s$, the entries $(a_{ij})_{1\leq j\leq n}$ in the $i^{th}$ row, gives a partition, of $n_i=\sum\limits_{j=1}^{n}ja_{ij}$, we can conclude that for the each of the first $d$ rows, since all entries are multiples of $r$, each row can be chosen in $p_r(n_i)$ number of ways. For $d+1\leq i \leq s$, we see that a row $(a_{ij})_{1\leq j\leq n}$ once again defines a partition of $n_i$ as explained above, but now this partition has the property that each part which is a multiple of $r$ occurs with multiplicity divisible by $r$. This is made sure by condition (3), of the specified combinatorial data. Thus, for each $d+1\leq i\leq s$, each $i^{th}$ row can be chosen in $p_r^{'}(n_i)$ number of ways. Finally observe that, $\sum\limits_{i=1}^{s}n_i=n$, by condition (1) of the combinatorial data. Thus, by the principle of multiplication, we conclude that the number of combinatorial data that satisfies the three conditions specified above (which is equal to the number of conjugacy classes that are  $r^{th}$ powers) is,
		$$\sum\limits_{n_1,n_2,\ldots,n_s}p_r(n_1)\ldots p_r(n_d)p_r^{'}(n_{d+1})\ldots p_r^{'}(n_s)$$
		where the above sum runs over all ordered pairs $(n_1,n_2,\ldots, n_s)$ satisfying $n_i\geq 0$ for every $1\leq i\leq s$ and $\sum\limits_{i=1}^{s}n_i=n$.
	\end{proof}
	
	\begin{example}
		Let us count the number of conjugacy classes that are squares in $S_3\wr S_3$. So, $r=2$, $n=3$. Since, $S_3$ has a total of 3 conjugacy classes, out of which exactly, two are squares, we have, $s=3, d=1$. We have the following table for $p(n), p_2(n), p_2^{'}(n)$ for $n=1,2,3$.
		
		\begin{table}[ht]
			\begin{tabular}{ |c|c|c|c| } 
				\hline
				$n$ & $p(n)$ & $p_2(n)$ & $p_2^{'}(n)$\\
				\hline
				1 & 1 & 0 & 1\\
				\hline
				2 & 2 & 1 & 1\\
				\hline
				3 & 3 & 0 & 2\\
				\hline
			\end{tabular}
			\vspace{0.1 cm}
			\caption{\label{table 2}$p(n), p_2(n), p_2^{'}(n)$ for $n=1,2,3$}
		\end{table}
	\noindent Using the formula in Theorem~\ref{powerconjformula}, and using the above table we have,
	$$\mathcal{CC}_2(S_3\wr S_3)=p_2(3)p_2^{'}(0)p_2^{'}(0)+p_2(0)p_2^{'}(3)p_2^{'}(0)+p_2(0)p_2^{'}(0)p_2^{'}(3) +p_2(2)p_2^{'}(1)p_2^{'}(0)+$$ $$p_2(2)p_2^{'}(0)p_2^{'}(1)+p_2(1)p_2^{'}(2)p_2^{'}(0)+p_2(1)p_2^{'}(0)p_2^{'}(2)+p_2(0)p_2^{'}(1)p_2^{'}(2)+p_2(0)p_2^{'}(2)p_2^{'}(1) $$$$+p_2(1)p_2^{'}(1)p_2^{'}(1)=8$$

	\noindent which agrees with our direct calculation in Example~\ref{ex3}.
	
	\end{example}

	We write the generating function for $\mathcal{CC}_r(G\wr S_n)$. The generating function for $p_r(n)$ and $p^{'}_r(n)$ is easy to deduce.
	
	\begin{equation}\label{eqnp_r}
		P_r(u)=1+\sum\limits_{n=1}^{\infty}p_r(n)u^n=\prod\limits_{i=1}^{\infty} \frac{1}{1-u^{ri}}=P(u^r).
	\end{equation}
	
	\noindent The last equality comes from Equation~\ref{partgen}.
	\begin{equation}\label{eqnp_r1}
		P_r^{'}(u)=1+\sum\limits_{n=1}^{\infty}p_r^{'}(n)u^n=\prod\limits_{i=1}^{\infty}\frac{1}{(1-u^{ri-1})\ldots(1-u^{ri-(r-1)})(1-u^{r^2i})}.
	\end{equation}
	 
    \noindent Using the formula for $\mathcal{CC}_r(G\wr S_n)$ in the previous theorem, and by Cauchy product of formal power series, we conclude that,
	\begin{proposition}\label{powerconj-genfunction}
		Let $r\geq 2$ be a prime and $G$ be a finite group. Suppose, out of the total $s$ number of conjugacy classes of $G$, exactly $d$ of them are not $r^{th}$ powers. Then, 
		\begin{center}
				$1+\sum\limits_{n=1}^{\infty}\mathcal{CC}_r(G\wr S_n)u^n=P(u^r)^{d}P_r^{'}(u)^{s-d}.$
		\end{center}
	\end{proposition}
	
	\subsection{Questions on asymptotics of the power map on $G\wr S_n$}
	\noindent We end this section by briefly discussing some asymptotic aspects of the powers in $G\wr S_n$. It is clear that, for $r\geq 2$, 
	\begin{center}
		$|\omega_r(G\wr S_n)|\leq |G|^n|\omega_r(S_n)| \implies P_r(G\wr S_n)\leq \frac{|\omega_r(S_n)|}{n!}=P_r(S_n).$
	\end{center}
	From the above inequality it is clear that for a fixed $G$ and $r\geq 2$ prime, 
	\begin{center}
		$\lim\limits_{n\to \infty}P_r(G\wr S_n)=0$
	\end{center}
	due to Sandwich theorem, and the fact that $\lim\limits_{n\to \infty}P_r(S_n)=0$ (see \cite{bmw}).
	
	\noindent It will be interesting to find estimates for powers in $G\wr S_n$ analogous to those in $S_n$ (see section 2). We finish this paper with an interesting question on the distribution of the values $P_r(G\wr S_n)$  for various $G$, which we observed using some GAP calculations on the powers. 
	
	\begin{question}
		Fix $n\geq 2$. Let $r\geq 2$ be a prime. Suppose $G$ be a finite group, with $(r,|G|)=1$. Is it true that
		\begin{center}
			$P_r(S_{n+1})\leq P_r(G\wr S_n)\leq P_r(S_n)$ 
		\end{center}
		when $n\equiv -1(\text{mod }r)$?
	\end{question}
	
	This means that for a fix $n$, the values $P_r(G\wr S_n)$ are distributed between $P_r(S_{n+1})$ and $P_r(S_n)$. In fact, this compels us to ask another natural question.
	
	\begin{question}
		Does the following result hold: For a fixed $n$, with $n\equiv-1(\text{mod } r)$ and given $\varepsilon >0$, there exists $N\in \mathbb{N}$, such that
		\begin{center}
			$P_r(G\wr S_n)-P_r(S_{n+1}) <\varepsilon$
		\end{center}
		for all $G$, with $|G|>N$ and, $r\nmid |G|$. 
	\end{question}
	This in fact says that as size of $G$ grows larger, the value of $P_r(G\wr S_n)$ converges to $P_r(S_{n+1})$. The fact that every element is a $r^{th}$ power in $G$, when $r\nmid |G|$ in some sense controls the growth of $P_r(G\wr S_n)$, hence giving a strong conviction that the above results can be answered in positive. 
	
	\begin{remark}
		The above result, definitely doesn't hold if we remove the condition that $r\nmid |G|$. Suppose $G=C_2$, the cyclic group of order 2. Consider $r=2$. Take $n=3$. We have, $P_2(S_3)=\frac{1}{2}=P_2(S_4)$. But, $P_2(C_2\wr S_3)=\frac{1}{4}<\frac{1}{2}$. 
	\end{remark}

\end{document}